\documentclass[12pt]{article}
\usepackage[utf8]{inputenc}
\usepackage{enumitem}
\usepackage{amsmath}
\usepackage{amssymb}
\usepackage{booktabs}
\usepackage{setspace}
\usepackage{graphicx}
\usepackage{amsthm}
\usepackage{float}
\usepackage{subcaption}
\usepackage{hyperref}
\hypersetup{breaklinks=true}
\usepackage{geometry}
\usepackage[affil-it]{authblk}
\usepackage{authblk}

\usepackage[authoryear,round]{natbib}
\makeatletter 
\renewcommand\NAT@biblabel[1]{[#1]}

\newtheorem{thm}{Theorem}

\newtheorem{as}{Assumption}

\newcommand{\fn}{\frac{1}{n}}
\newcommand{\betahat}{\hat{\beta}}
\newcommand{\sigmahat}{\hat{\sigma}}

\newcommand{\argmin}{\mathop{\rm argmin}\limits}

\newcommand*{\transpose}{%
  {\mathpalette\@transpose{}}%
  }
  
  \expandafter\let\expandafter\oldproof\csname\string\proof\endcsname
\let\oldendproof\endproof
\renewenvironment{proof}[1][\proofname]{%
  \oldproof[\bfseries \scshape #1]%
}{\oldendproof}

\title{Testing Heteroskedasticity in High-Dimensional Linear Regression}
\author{Akira Shinkyu\thanks{Email: akira.shinkyu.5589@gmail.com}}
\affil{Kobe University\thanks{Graduate School of Economics, 2-1 Rokkodai-cho, Nada-ku, Kobe, 657-8501, JAPAN.}}
\providecommand{\keywords}[1]{\textbf{Keywords}: #1}
\date{\today}
\allowdisplaybreaks
\begin{document}
\maketitle
\begin{abstract}
We propose a new testing procedure of heteroskedasticity in high-dimensional linear regression, where the number of covariates can be larger than the sample size. Our testing procedure is based on residuals of the Lasso. We demonstrate that our test statistic has asymptotic normality under the null hypothesis of homoskedasticity. Simulation results show that the proposed testing procedure obtains accurate empirical sizes and powers. We also present results of real economic data applications. 
\end{abstract}
\section{Introduction}
Many testing procedures of heteroskedasticity for linear regression models have been proposed in econometrics and statistics since the 1960s. For example, in econometrics, see \cite{h1976}, \cite{s1978}, \cite{g1978}, \cite{bp1979}, \cite{w1980}, \cite{k1981}, \cite{kb1982}, and \cite{np1987}. On the other hand, in statistics, see \cite{gq1965}, \cite{g1969}, \cite{r1969}, \cite{b1978}, \cite{hm1979}, \cite{cr1981}, \cite{cw1983}, and \cite{db1997}. \par
Recently, with the advancement of computer technology, high-dimensional data have been extensively collected. As a result, high-dimensional linear regression models, where the number of covariates $p$ is large compared with the sample size $n$, have become more prevalent, attracting significant attention in recent econometrics and statistics studies. Because the testing methods of heteroskedasticity listed above were proposed in the asymptotic framework, where $p$ is much smaller than $n$, the validity of inference results by these methods cannot be guaranteed when $p$ is also large compared with $n$. \par 

To conduct valid heteroskedasticity tests in linear regression when $p$ is large compared with $n$, \cite{ly2019} proposed heteroskedasticity testing procedures called the approximate likelihood ratio test (ALRT) and the coefficient of variation test (CVT). By employing the asymptotic setting where $p<n$ and $n-p\rightarrow \infty$, they demonstrated that these test statistics have asymptotic normality under the null hypothesis of homoskedasticity. Their simulation studies also showed that these two methods have better empirical sizes and powers than classical testing methods such as the White test and the Breusch-Pagan test. \par

However, ALRT and CVT also depend on OLS residuals. Therefore, these methods are not available when $p>n$ because the design matrix is not full column rank. Furthermore, according to the simulation studies of \cite{ly2019}, even if $p<n$ but $p$ is close to $n$, the empirical powers of these methods are too low because OLS is unstable due to the near singularity of the Gram matrix of covariates. As far as we know, there are no existing methods of testing heteroskedasticity when $p>n$ and performs stably when $p<n$ but $p$ is close to $n$.\par 

In this study, we propose a testing procedure of heteroskedasticity for high-dimensional linear regression models. We replace OLS residuals of CVT with residuals of the Lasso (\citealt{t1996}). Therefore, we call the proposed testing method a Lasso-based CVT (LCVT). We show that the test statistic has asymptotic normality under the null hypothesis of homoskedasticity even when $p>n$. Simulation results show that LCVT obtains accurate empirical sizes and powers when $p>n$ and performs more stably than ALRT and CVT when $p<n$ and $p$ is close to $n$ due to the $\ell_1$ regularization. We also report the results of real economic data applications by LCVT. The implementation of our method is very simple because of the standard R package by \cite{fht2010}. \par 

The remainder of our paper is organized as follows. Section 2 introduces a high-dimensional linear regression model and proposes a test statistic for the null hypothesis of homoskedasticity. Section 3 presents assumptions and asymptotic properties of the test statistic. Section 4 presents the results of our simulation studies. Section 5 reports the results of real data applications. Section 6 concludes the study, and the appendix provides a proof of the main theorem.\par 

\section{Model and Test Statistic}
First, we introduce some notations. Let $a^T$ be the transpose of a vector $a$. Let $\|a\|$, $\|a\|_1$, and $\|a\|_\infty$ be the $\ell_2$ norm, $\ell_1$ norm, and maximum element in the absolute value of a vector $a$, respectively. For any sequence of positive real numbers $\{a_n\}$ and $\{b_n\}$, $a_n\asymp b_n$ denotes $C_1 \leq a_n/b_n\leq C_2$ for all $n$ and some positive constants $C_1$ and $C_2$. Let $\lambda_{\min}(A)$ be the minimum eigenvalue of the matrix $A$. Let $C$ be a generic positive constant that can vary from line to line. \par 
We consider the linear regression model
\begin{equation}
    Y_i=X_i^T\beta_0+\epsilon_i, \label{79}
\end{equation}
where $Y_i\in \mathbb{R}$ denotes the response, $X_i=(X_{i,1},\ldots,X_{i,p})^T\in \mathbb{R}^p$ denotes the random covariate, $\beta_0 \in \mathbb{R}^p$ denotes the unknown regression coefficient, and $\epsilon_i \in \mathbb{R}$ denotes the unobservable error with mean zero and conditional variance $\sigma(X_i)^2$ for $i=1,\ldots,n$ and some positive function $\sigma(\cdot)$. Let $s_0$ be the number of nonzero components in $\beta_0$. We also denote $Y=(Y_1,\ldots,Y_n)^T\in \mathbb{R}^n$, $X=(X_1,\ldots,X_n)^T\in \mathbb{R}^{n \times p}$, and $\epsilon=(\epsilon_1,\ldots,\epsilon_n)^T \in \mathbb{R}^n$. Then, the matrix form of \eqref{79} is expressed as $Y=X\beta_0+\epsilon$. Throughout the paper, we assume that $p>\kappa n$ for some $0<\kappa<1$ and $\log p/n =o(1)$.\par

The purpose of our study is to investigate whether $\sigma(\cdot)$ is a positive constant or a positive function of $X_i$. To this aim, we test the null hypothesis of homoskedasticity:
\begin{equation}
    H_0: \sigma(\cdot)=\sigma_0, \label{72} 
\end{equation}
where $\sigma_0$ denotes some positive constant.\par 
To derive the test statistic for the null hypothesis of homoskedasticity in the high-dimensional linear regression model, we first estimate the coefficient $\beta_0$ by the Lasso:
\begin{equation*}
    \betahat:=\argmin_{\beta \in \mathbb{R}^p}\frac{1}{n}\|Y-X\beta\|^2+2\lambda_0\|\beta\|_1,
\end{equation*}
where $\lambda_0>0$ represents a tuning parameter. The Lasso is one of the most popular and successful methods for estimating the coefficients of high-dimensional linear regression models. An excellent overview of the Lasso is given by \cite{bv2011}. This estimator has variable selection properties, and useful convergence rates are available due to oracle inequalities (\citealt{brt2009}). Furthermore, it can be unique even when $p>n$ under a general condition (\citealt{t2013}). \par 
We construct a test statistic for the null hypothesis of homoskedasticity based on residuals of the Lasso. Define the residuals and the sample mean of squared residuals as 
\begin{align*}
    \hat{\epsilon}_i&=Y_i-X_i^T\betahat, \\ \sigmahat^2&=\frac{1}{n}\|Y-X\betahat\|^2.
\end{align*}
We propose the following coefficient of variation statistic as a test statistic for the null hypothesis of homoskedasticity:
\begin{equation}
    T=\frac{1}{n\sigmahat^4} \sum_{i=1}^n (\hat{\epsilon}_i^2-\sigmahat^2)^2. \label{99}
\end{equation}
The idea of the coefficient of variation comes from \cite{j1971}, who proposed the test statistics for the null hypothesis of spherical covariance matrices. 
\section{Theoretical Results}
We present assumptions to investigate the properties of the proposed test statistic for the null hypothesis of homoskedasticity. Roughly speaking, our test statistic has asymptotic normality for high-dimensional linear regression models under the null hypothesis if covariates and errors are Gaussian.\par 
First, we make the following assumption on errors. 
\begin{as}
\textup{Errors $\epsilon_i$, $i=1,\ldots,n$, can be expressed as $\epsilon_i=\sigma(X_i) \xi_i$, where $\xi_i$ independently and identically follows a standard normal distribution. Furthermore, $\xi=(\xi_1,\ldots,\xi_n)^T$ is independent of $X$. 
}
\end{as}
This assumption implies that errors are conditional Gaussian: $\epsilon_i | X_i \sim N(0, \sigma(X_i)^2)$ for $i=1,\ldots,n$. Furthermore, this assumption also implies that under the null hypothesis of homoskedasticity $\epsilon$ follows a multivariate normal distribution with mean zero and covariance matrix $\sigma^2_0 I_n$, and it is independent of $X$. This assumption is important for deriving the order of the Lasso and the asymptotic joint distribution of the fourth and second sample moments of errors under the null hypothesis of homoskedasticity.  \par 
Next, we make the following assumption on covariates.
\begin{as}
\textup{Covariates $X_i$, $i=1,\ldots,n$, independently and identically follow a multivariate normal distribution with mean zero and covariance matrix $\Sigma$. Furthermore, there exist some positive constants $C_{\min}$ and $C_{\max}<\infty$ such that $C_{\min} \leq \lambda_{\min}(\Sigma)$ and $\max_{j} \Sigma_{j,j}<C_{\max}.$ 
}
\end{as}
There is a similar assumption in (A2) of \cite{vbrd2014}. Under this assumption and further assuming $s_0=o(n/\log p)$, the so-called compatibility condition (\citealt{bv2011}), which is important to derive the order of the Lasso, holds with probability tending to one. See Lemma 5.2 of \cite{vbrd2014} and Theorem 1 of \cite{rwy2010}.\par 
Under these assumptions, we can demonstrate that asymptotic normality holds for the proposed test statistic under the null hypothesis of homoskedasticity.
\begin{thm}
\label{thm1}
Suppose that Assumptions 1 and 2 hold and $s_0=o(\sqrt{n}/(\log p)^2)$. Let $\lambda_0 \asymp \sqrt{\log p/n}$ be a suitably chosen tuning parameter. Then, under $H_0$, we obtain
\begin{equation}
    \sqrt{n}(T-2)\xrightarrow{d}N(0,24). \label{118}
\end{equation}
\end{thm}
We call the testing procedure using the test statistic $T$ and critical values from the normal distribution in \eqref{118} LCVT. \par    
Because our test statistic depends on the Lasso, our method requires sparsity of $\beta_0$. Compared with the sparsity level $s_0=o(\sqrt{n}/\log p)$, which is required by famous statistical inference methods for components of $\beta_0$ such as the debiased Lasso (\citealt{vbrd2014}, \citealt{zz2014}, \citealt{jm2014}) and the post-double selection (\citealt{bch2014}), our method requires a slightly stronger sparsity level: $s_0=o(\sqrt{n}/(\log p)^2)$. This sparsity level is crucial to evaluate the sample mean of fourth-powered residuals.

\section{Simulation Studies}
In this section, we present the simulation results. The purpose of the simulation studies is to conduct the 5\% significance level hypothesis test for the null hypothesis of homoskedasticity. We calculate the rejection rates of the null hypothesis by repeating the experiments 1,000 times. \par 
We consider two settings. One is a high-dimensional setting: $p=2n$. The other is a low-dimensional setting: $p=0.9n$. For the high-dimensional setting, we investigate the empirical sizes and powers of LCVT. For the low-dimensional setting, we compare the empirical sizes and powers of LCVT with those of ALRT and CVT. \par 
To use LCVT, we implement the Lasso using the R package \verb|glmnet|. The tuning parameter of the Lasso $\lambda_0$ is selected by the one standard error rule with the \verb|lambda.1se| option. Candidates of $\lambda_0$ are also generated by \verb|glmnet|.
\subsection{Simulation Design}
The covariance matrix of covariates is provided by one of the following matrices, as in \cite{whx2020}. 
\begin{align*}
    \text{Independent}&: \Sigma=I,\\  \text{Toeplitz}&: \Sigma_{i,j}=0.9^{|i-j|},\\   \text{Equi correlation}&: \Sigma_{i,j}=0.3^{1(i \neq j)}\ \text{or}\ \ \Sigma_{i,j}=0.9^{1(i \neq j)},
\end{align*}
where $1(\cdot)$ represents the indicator function. The data generating process is given by
\begin{equation*}
        Y_i=X_{i}^T\beta_0+\epsilon_i,
\end{equation*}
where $X_i \in \mathbb{R}^{p} \sim N(0, \Sigma)$ and $\epsilon_i|X_i \sim  N(0,\sigma(X_i)^2)$ for $i=1.\ldots,n$. We set $\sigma(\cdot)=1$ for homoskedastic cases. Forms of the function $\sigma(\cdot)$ for heteroskedastic cases and the coefficient $\beta_0$ are defined later. \par 
In the high-dimensional setting, we consider a setting similar to Setting 1 of \cite{whx2020}. For a given value $\tau$, let $\beta_0$ be one of the following $p$ dimensional vectors.
    \begin{align*}
    \text{sparse}&:(1, \underbrace{\tau,\ldots,\tau}_{4},0,\ldots,0)^T,\\
    \text{moderately sparse}&:(1,\underbrace{5\tau,\ldots,5\tau}_{9},\underbrace{\tau,\ldots,\tau}_{10},0,\ldots,0)^T, \\
    \text{dense}&: \left(1,\frac{\tau}{\sqrt{2}},\ldots,\frac{\tau}{\sqrt{p}}\right)^T.
\end{align*}
We choose $\tau$ so that $R^2$ is 0.8 when conditional error variances are homoskedastic. We also choose the same $\tau$ when conditional error variances are heteroskedastic.\par

To investigate the empirical powers of LCVT in the high-dimensional setting, we consider the following forms of the function $\sigma(\cdot)$ for heteroskedasticity.
\begin{itemize}
    \item Form 1: $\sigma(X_i)=\exp(\rho_0  X_i^T\beta_0)$ with $\rho_0=0.4$,
    \item Form 2: $\sigma(X_i)=\sqrt{\frac{(1+X_i^T\beta_0)^2}{\fn \sum_{i=1}^n (1+X_i^T\beta_0)^2}}$,
    \item Form 3: $\sigma(X_i)=(1+\sum_{j=1}^p\sin(10 X_{i,j})/j)^2$.
\end{itemize}
We refer to \cite{cw1983} for Form 1. We borrow Form 2 from \cite{bch2014}. We also refer to \cite{dm1998} for Form 3. \par     
In the low-dimensional setting, we consider a setting similar to \cite{ly2019}. $\beta_0$ is given by the following $p$ dimensional vector. 
\begin{equation*}
    \beta_0=(\underbrace{1,\ldots,1}_{\frac{p}{2}},0,\ldots,0).
\end{equation*}
To investigate the empirical powers of testing procedures in the low-dimensional setting, we consider the following forms of the function $\sigma(\cdot)$ for heteroskedasticity, as in \cite{ly2019}. 
\begin{itemize}
    \item Form 4: $\sigma(X_i)=\exp(0.5 \sum_{j=1}^{\frac{p}{10}}X_{i,j})$,
    \item Form 5: $\sigma(X_i)=(1+0.5\sum_{j=1}^{\frac{p}{10}}\sin(10X_{i,j}))^2$,
    \item Form 6: $\sigma(X_i)=(1+0.5\sum_{j=1}^{\frac{p}{10}}X_{i,j})^2$.
\end{itemize}
Under these settings, we report the rejection rates of the null hypothesis by testing procedures. Note that when the null hypothesis is true, rejection rates indicate empirical sizes. Conversely, when the alternative hypothesis is true, rejection rates indicate empirical powers. Although we also calculated standard deviations of rejection rates, we do not report them here to simplify the presentation of the simulation results.\par 
\subsection{Simulation Results}
First, we check the simulation results in the high-dimensional setting. Table \ref{t156} shows the rejection rates of LCVT when conditional error variances are homoskedastic in the high-dimensional cases. When $n=100$ and $p=200$, the empirical sizes obtained by LCVT are slightly smaller than the nominal size level. This indicates that the inference results obtained by our method may be slightly conservative when the sample size is small. However, the empirical sizes of LCVT do not significantly deviate from the nominal size level. In addition, there are no significant differences in rejection rates among covariance and coefficient structures. When $n=500$ and $p=1000$, the empirical sizes of LCVT are close to the nominal size level among all covariance and coefficient structures. Therefore, it is expected that LCVT can control sizes when the sample size is large.\par 

Table \ref{t180} shows the rejection rates of LCVT when conditional error variances are heteroskedastic and $n=100$ and $p=200$. When covariates are independent, the empirical powers of LCVT tend to be low. This is because the variation in conditional error variances is small in this case, causing the test statistic to tend to be small. On the other hand, when correlations occur among covariates, the empirical powers of LCVT are high. Correlations increase the variation in conditional error variances, making it easier for LCVT to detect heteroskedasticity. \par

Table \ref{t227} shows the rejection rates of LCVT when conditional error variances are heteroskedastic and $n=500$ and $p=1000$. These results show that LCVT obtains high empirical powers among all coefficient and covariance structures. Therefore, we expect that LCVT will have high powers when the sample size is large. \par 

Next, we check simulation results in the low-dimensional setting. Table \ref{t245} shows the rejection rates of LCVT, ALRT, and CVT when conditional error variances are homoskedastic. When $n=100$ and $p=90$, the empirical sizes of ALRT are accurate. On the other hand, LCVT and CVT produce slightly conservative results. When $n=500$ and $p=450$, ALRT is successful in size control. The large sample size slightly increases the rejection rates of LCVT, and they are more accurate than those of CVT.\par

Table \ref{t289} shows the rejection rates of LCVT, ALRT, and CVT when conditional error variances are heteroskedastic and $n=100$ and $p=90$. LCVT generally tends to yield high empirical powers even though it is conservative when covariates are independent or have a low correlation in Form 5. On the other hand, the empirical powers of ALRT and CVT are too low. This is because OLS estimates are unstable when $p$ is close to $n$ due to the near singularity of the sample covariance.\par 

Table \ref{t328} presents the rejection rates of LCVT, ALRT, and CVT when conditional error variances are heteroskedastic and $n=500$ and $p=450$. The large sample size increases the empirical powers of LCVT, which correctly rejects the null hypothesis in general. On the other hand, the results of ALRT and CVT are too conservative even when the sample size is large. ALRT generally fails to reject the null hypothesis. CVT obtains high empirical powers in Form 4, but its empirical powers are too low in the other forms of heteroskedasticity. These results indicate that a large sample size may not improve the powers of OLS-based testing methods when $p$ is close to $n$. \par  

In conclusion, the simulation studies reveal the following. In the high-dimensional setting, overall, LCVT obtains accurate empirical sizes and powers. In the low-dimensional setting, LCVT can also control the empirical sizes and obtains significantly better empirical powers than ALRT and CVT when $p$ is close to $n$. Therefore, LCVT would be promising for detecting heteroskedasticity when $p>n$ and a reliable alternative to OLS-based testing methods when $p<n$ and $p$ is close to $n$. 

\begin{table}
\centering 
\caption{Rejection rates of LCVT when $H_0$ is true and $p=2n$. }
\begin{tabular}{ccc}
\midrule
Covariance    & $n=100$                  & $n=500$                \\ \midrule
              & \multicolumn{2}{c}{$\beta_0$ is sparse}            \\ \cmidrule{2-3} 
Independent   & 0.031                  & 0.052                \\
Equi corr 0.3 & 0.039                  & 0.045                \\
Equi corr 0.9 & 0.037                  & 0.055                \\
Toeplitz 0.9  & 0.036                  & 0.04                 \\ \cmidrule{2-3} 
              & \multicolumn{2}{c}{$\beta_0$ is moderately sparse} \\ \cmidrule{2-3} 
Independent   & 0.026                  & 0.045                \\
Equi corr 0.3 & 0.04                   & 0.039                \\
Equi corr 0.9 & 0.034                  & 0.051                \\
Toeplitz 0.9  & 0.037                  & 0.038                \\ \cmidrule{2-3} 
              & \multicolumn{2}{c}{$\beta_0$ is dense}             \\ \cmidrule{2-3} 
Independent   & 0.037                  & 0.047                \\
Equi corr 0.3 & 0.032                  & 0.036                \\
Equi corr 0.9 & 0.033                  & 0.045                \\
Toeplitz 0.9  & 0.033                  & 0.037                \\ \midrule
\end{tabular}
\label{t156}
\end{table}

\begin{table}
\centering 
\caption{Rejection rates of LCVT when $H_1$ is true and $n=100$ and $p=200$.}
\begin{tabular}{cccc}
\midrule
Covariance    & Form 1   & Form 2   & Form 3      \\ \midrule
              & \multicolumn{3}{c}{$\beta_0$ is sparse}            \\ \cmidrule{2-4} 
Independent   & 0.86     & 0.904     & 0.904           \\
Equi corr 0.3 & 0.873    & 0.881     & 0.948           \\
Equi corr 0.9 & 0.909    & 0.895     & 0.974            \\
Toeplitz 0.9  & 0.912    & 0.926     & 0.972           \\ \cmidrule{2-4} 
              & \multicolumn{3}{c}{$\beta_0$ is moderately sparse} \\ \cmidrule{2-4} 
Independent   & 0.803    & 0.757     & 0.882           \\
Equi corr 0.3 & 0.871    & 0.83      & 0.956           \\
Equi corr 0.9 & 0.913    & 0.908     & 0.976           \\
Toeplitz 0.9  & 0.92     & 0.925     & 0.96            \\ \cmidrule{2-4} 
              & \multicolumn{3}{c}{$\beta_0$ is dense}             \\ \cmidrule{2-4} 
Independent   & 0.738    & 0.546     & 0.855           \\
Equi corr 0.3 & 0.9      & 0.854     & 0.97            \\
Equi corr 0.9 & 0.907    & 0.906     & 0.978           \\
Toeplitz 0.9  & 0.864    & 0.78      & 0.947           \\ \midrule
\end{tabular}
\label{t180}
\end{table}

\begin{table}
\centering
\caption{Rejection rates of LCVT when $H_1$ is true and $n=500$ and $p=1000$. }
\begin{tabular}{cccc}
\midrule
Covariance    & Form 1    & Form 2   & Form 3      \\ \midrule
              & \multicolumn{3}{c}{$\beta_0$ is sparse}             \\ \cmidrule{2-4} 
Independent   & 1         & 1         & 1                   \\
Equi corr 0.3 & 1         & 1         & 1                   \\
Equi corr 0.9 & 1         & 1         & 1                   \\
Toeplitz 0.9  & 1         & 1         & 1                   \\ \cmidrule{2-4} 
              & \multicolumn{3}{c}{$\beta_0$ is moderately sparse} \\ \cmidrule{2-4} 
Independent   & 1         & 1         & 1                   \\
Equi corr 0.3 & 1         & 1         & 1                   \\
Equi corr 0.9 & 1         & 1         & 1                   \\
Toeplitz 0.9  & 1         & 1         & 1                   \\ \cmidrule{2-4} 
              & \multicolumn{3}{c}{$\beta_0$ is dense}              \\ \cmidrule{2-4} 
Independent   & 1         & 0.996     & 1                   \\
Equi corr 0.3 & 1         & 1         & 1                   \\
Equi corr 0.9 & 1         & 1         & 1                   \\
Toeplitz 0.9  & 1         & 1         & 1                   \\ \midrule
\end{tabular}
\label{t227}
\end{table}

\begin{table}
\centering
\caption{Rejection rates when $H_0$ is true and $p=0.9n$.}
\begin{tabular}{cccc}
\midrule
Covariance    & LCVT       & ALRT     & CVT      \\ \midrule
              & \multicolumn{3}{c}{$n=100$}  \\ \cmidrule{2-4} 
Independent   & 0.025      & 0.044    & 0.03     \\
Equi corr 0.3 & 0.028      & 0.044    & 0.03     \\
Equi corr 0.9 & 0.022      & 0.044    & 0.03     \\
Toeplitz 0.9  & 0.04       & 0.044    & 0.03     \\ \cmidrule{2-4} 
              & \multicolumn{3}{c}{$n=500$} \\ \cmidrule{2-4} 
Independent   & 0.05       & 0.05     & 0.041    \\
Equi corr 0.3 & 0.047      & 0.05     & 0.041    \\
Equi corr 0.9 & 0.055      & 0.05     & 0.041    \\
Toeplitz 0.9  & 0.044      & 0.05     & 0.041    \\ \midrule
\end{tabular}
\label{t245}
\end{table}

\begin{table}
\centering
\caption{Rejection rates when $H_1$ is true and $n=100$ and $p=90$.}
\begin{tabular}{cccc}
\midrule
Covariance    & LCVT    & ALRT    & CVT    \\ \midrule
              & \multicolumn{3}{c}{Form 4} \\ \cmidrule{2-4} 
Independent   & 0.819   & 0.052   & 0.081  \\
Equi corr 0.3 & 1       & 0.249   & 0.123  \\
Equi corr 0.9 & 1       & 0.969   & 0.163  \\
Toeplitz 0.9  & 1       & 0.899   & 0.145  \\ \cmidrule{2-4} 
              & \multicolumn{3}{c}{Form 5} \\ \cmidrule{2-4} 
Independent   & 0.289   & 0.047   & 0.043  \\
Equi corr 0.3 & 0.453   & 0.051   & 0.041  \\
Equi corr 0.9 & 0.626   & 0.057   & 0.046  \\
Toeplitz 0.9  & 0.814   & 0.057   & 0.043  \\ \cmidrule{2-4} 
              & \multicolumn{3}{c}{Form 6} \\ \cmidrule{2-4} 
Independent   & 0.667   & 0.054   & 0.053  \\
Equi corr 0.3 & 0.966   & 0.048   & 0.047  \\
Equi corr 0.9 & 0.998   & 0.058   & 0.038  \\
Toeplitz 0.9  & 0.989   & 0.052   & 0.039  \\ \midrule
\end{tabular}
\label{t289}
\end{table}

\begin{table}
\centering
\caption{Rejection rates when $H_1$ is true and $n=500$ and $p=450$.}
\begin{tabular}{cccc}
\midrule
Covariance    & LCVT    & ALRT    & CVT     \\ \midrule
              & \multicolumn{3}{c}{Form 4}  \\ \cmidrule{2-4} 
Independent   & 1       & 0.167       & 0.949   \\
Equi corr 0.3 & 1       & 0.243   & 0.998   \\
Equi corr 0.9 & 1       & 0.221   & 0.999   \\
Toeplitz 0.9  & 1       & 0.195   & 0.999   \\ \cmidrule{2-4} 
              & \multicolumn{3}{c}{Form 5}  \\ \cmidrule{2-4} 
Independent   & 0.992   & 0.056   & 0.203   \\
Equi corr 0.3 & 1       & 0.061   & 0.183   \\
Equi corr 0.9 & 1       & 0.061   & 0.188   \\
Toeplitz 0.9  & 1       & 0.059   & 0.185   \\ \cmidrule{2-4} 
              & \multicolumn{3}{c}{Form 6}  \\ \cmidrule{2-4} 
Independent   & 1       & 0.063   & 0.175   \\
Equi corr 0.3 & 1       & 0.053   & 0.186   \\
Equi corr 0.9 & 1       & 0.066   & 0.176 \\
Toeplitz 0.9  & 1       & 0.049   & 0.199   \\ \midrule
\end{tabular}
\label{t328}
\end{table}

\section{Real data example}

In this section, we apply testing methods of heteroskedasticity to real economic data. We revisit \cite{hr1978}, who investigated the demand for clean air using Boston housing data. To this aim, they considered the linear regression model
\begin{align*}
\log\text{MV}_i&= \alpha + \beta_1\text{NOX}_i^2 +\beta_2\text{RM}_i^2+\beta_3\log\text{DIS}_i+\beta_4\text{AGE}_i+\beta_5\log\text{RAD}_i+\beta_6\text{TAX}_i\\
&+\beta_7\text{PTRATIO}_i+\beta_8(B_i-0.63)^2+\beta_9\log\text{LSTAT}_i+\beta_{10}\text{CRIM}_i\\
&+\beta_{11}\text{ZN}_i+\beta_{12}\text{INDUS}_i +\beta_{13}\text{CHAS}_i +\epsilon_i,
\end{align*}
where $n=506$, and $\text{MV}_i$ denotes the median value of house prices. Explanations of the other variables are given in \cite{gp1996}. The dataset is available from the R package \verb|mlbench|. \par

In this example, we investigate the sensitivity of the testing results by LCVT, ALRT, and CVT by adding an irrelevant artificial covariate. First, we applied these testing methods to the original model. The studentized test statistics of LCVT, ALRT, and CVT are 15.982, 4.353, and 14, respectively. Therefore, all methods strongly rejected the null hypothesis of homoskedasticity with a 5\% significance level, and it is reasonable to believe that conditional error variances in the original model are heteroskedastic.  \par 

Next, we add the irrelevant artificial covariate $W_i \in \mathbb{R}^d \sim N(0, \Sigma)$ to the original model:
\begin{equation}
     \log\text{MV}_i=\alpha+X_i^T\beta+W_i^T\theta +\epsilon_i, \label{340}
\end{equation}
where the covariance matrix is Toeplitz: $\Sigma_{i,j}=0.9^{|i-j|}$, $d$ is specified in Table \ref{t354}, and $X_i$ includes the 13 variables in the original model. Because $W_i$ is irrelevant, we know that the true $\theta$ is zero. Therefore, we also know that conditional error variances in the model \eqref{340} are heteroskedastic. \par

Table \ref{t354} shows the rejection rates for the null hypothesis of homoskedasticity in the model \eqref{340} over 1,000 repetitions. As the number of irrelevant covariates increases, the rejection rates of ALRT and CVT decrease. Similar to the simulation results in Section 4, it is difficult for ALRT to reject the null hypothesis when there are many covariates. Although inference results of CVT are more robust to the number of covariates than ALRT, its rejection rates decrease significantly when the number of irrelevant covariates is greater than 200. On the other hand, the rejection rates of LCVT are high even when the number of irrelevant covariates is significantly greater than the sample size. This is because the Lasso selects many relevant covariates and omits many irrelevant covariates. Therefore, the test statistic of LCVT in the model \eqref{340} does not deviate from that in the original model, so LCVT rejects the null hypothesis of homoskedasticity. \par

This real data example shows that the inference results of LCVT are robust to irrelevant covariates. On the other hand, the inference results of ALRT and CVT are sensitive to irrelevant covariates. These results imply that LCVT is useful for heteroskedasticity detection, particularly when linear regression models possibly have many irrelevant covariates. 
\begin{table}
\centering
\caption{Rejection rates for the null hypothesis of homoskedasticity in the model \eqref{340} with $n=506$ over 1,000 repetitions.}
\begin{tabular}{ccccccccc}
\midrule
$d$    & 100   & 200   & 300   & 400   & 500 & 1000 & 2000 & 3000 \\ \midrule
LCVT & 1     & 1     & 1     & 1     & 1   & 1    & 1  &1  \\
ALRT & 0.877 & 0.443 & 0.15 & 0.056 & -   & -    & - &-    \\
CVT  & 1     & 0.973 & 0.49 & 0.098 & -   & -    & - &-   \\ \midrule
\end{tabular}
\label{t354}
\end{table}

\section{Conclusion}
In this study, we proposed a testing procedure of heteroskedasticity called LCVT for high-dimensional linear regression models. LCVT is derived from residuals of the Lasso. Therefore, LCVT is available even when the number of covariates is larger than the sample size. Under normality assumptions on errors and covariates, we demonstrated that our test statistic has asymptotic normality under the null hypothesis of homoskedasticity. Simulation studies and real data applications show that LCVT performs well. Therefore, we believe that LCVT is an effective method for detecting heteroskedasticity in high-dimensional linear regression.

\section*{Acknowledgments}
The author would like to thank Professor Naoya Sueishi for his valuable comments, which give significant improvement to the presentation. This research was supported by JST SPRING Grant Number JPMJFS2126.

\section*{Appendix}
\begin{proof}[Proof of Theorem \ref{thm1}]
Under Assumption 1 and the null hypothesis of homoskedasticity, the Lindeberg-Levy Central Limit Theorem yields  
\begin{equation*}
    \sqrt{n}\begin{pmatrix}
    \frac{1}{n}\sum_{i=1}^n\epsilon_i^4-3\sigma_0^4 \\
    \frac{1}{n}\sum_{i=1}^n\epsilon_i^2-\sigma_0^2
\end{pmatrix}
\xrightarrow{d}N(0, \Sigma_0),
\end{equation*}
where
\begin{equation*}
    \Sigma_0=\begin{pmatrix}
    \text{Var}(\epsilon_i^4) & \text{Cov}(\epsilon_i^4, \epsilon_i^2) \\
    \text{Cov}(\epsilon_i^4, \epsilon_i^2) & \text{Var}(\epsilon_i^2)
    \end{pmatrix}=\begin{pmatrix}
    96\sigma_0^8 & 12\sigma_0^6 \\
    12\sigma_0^6 & 2 \sigma_0^4
    \end{pmatrix}.
\end{equation*}
For $\theta=(\theta_1,\theta_2)^T$, let $f(\theta)=\theta_1/\theta_2^2-1$. Then, partial derivatives are given by
\begin{equation*}
\frac{\partial f(\theta)}{\partial \theta_1}=\frac{1}{\theta_2^2},\ \  \frac{\partial f(\theta)}{\partial \theta_2}=-\frac{2\theta_1}{\theta_2^3}.  
\end{equation*}
Let
\begin{equation*}
 \hat{\theta}=\left(\frac{1}{n} \sum_{i=1}^n \epsilon_i^4, \frac{1}{n}\sum_{i=1}^n \epsilon_i^2\right)^T,\ \ \theta_0=(3\sigma_0^4,\sigma_0^2)^T.   
\end{equation*} 
Then, the delta method yields \begin{equation}
    \sqrt{n}(f(\hat{\theta})-f(\theta_0))\xrightarrow{d}N \left(0, \frac{\partial f(\theta_0)}{\partial \theta ^T} \Sigma_0 \frac{\partial f(\theta_0)}{\partial \theta} \right) \label{203},
\end{equation}
where
\begin{equation}
    f(\hat{\theta})=\frac{\frac{1}{n}\sum_{i=1}^n \epsilon_i^4}{(\frac{1}{n}\sum_{i=1}^n \epsilon_i^2)^2}-1,\ \  f(\theta_0)=\frac{3\sigma_0^4}{(\sigma_0^2)^2}-1=2, \label{205}
\end{equation}
\begin{equation}
    \frac{\partial f (\theta_0)}{\partial \theta ^T} \Sigma_0 \frac{\partial f(\theta_0)}{\partial \theta}=\left(\frac{1}{\sigma_0^4}, \frac{-6 }{\sigma_0^2} \right)\begin{pmatrix}
    96\sigma_0^8 & 12\sigma_0^6 \\
    12\sigma_0^6 & 2 \sigma_0^4
    \end{pmatrix}
    \begin{pmatrix}
    \frac{1}{\sigma_0^4}\\
    \frac{-6}{\sigma_0^2}
    \end{pmatrix}=24. \label{217}
\end{equation}
Therefore, it is sufficient to show that 
\begin{equation*}
   \sqrt{n}(T-2)=\sqrt{n}(f(\hat{\theta})-f(\theta_0))+o_p(1).
\end{equation*}
We know that
\begin{equation*}
    \hat{\epsilon}_i^2=\epsilon_i^2-2\epsilon_iX_i^T(\betahat-\beta_0)+\{X_i^T(\betahat-\beta_0)\}^2,
\end{equation*}
so we obtain
\begin{equation}
    \frac{1}{n}\sum_{i=1}^n\hat{\epsilon}_i^2= \frac{1}{n}\sum_{i=1}^n\epsilon_i^2- \frac{2}{n}\sum_{i=1}^n\epsilon_iX_i^T(\betahat-\beta_0)+ \frac{1}{n}\sum_{i=1}^n\{X_i^T(\betahat-\beta_0)\}^2, \label{232}
\end{equation}
and 
\begin{align}
    \frac{1}{n}\sum_{i=1}^n\hat{\epsilon}_i^4&= \frac{1}{n}\sum_{i=1}^n\epsilon_i^4+ \frac{6}{n}\sum_{i=1}^n\{\epsilon_iX_i^T(\betahat-\beta_0)\}^2+ \frac{1}{n}\sum_{i=1}^n\{X_i^T(\betahat-\beta_0)\}^4 \nonumber \\
    &- \frac{4}{n}\sum_{i=1}^n\epsilon_i^3X_i^T(\betahat-\beta_0)- \frac{4}{n}\sum_{i=1}^n\epsilon_i\{X_i^T(\betahat-\beta_0)\}^3. \label{238}
\end{align} 
First, we evaluate the sample mean of squared residuals \eqref{232}. Let $C_0$ be a sufficiently large positive constant satisfying $C_0 \geq \sigma_0 \sqrt{C_{\max}}$, and set $\lambda_0=C_0\sqrt{2 \log p/n}$. Recall that $C_{\max}$ satisfies $\max_{j}\Sigma_{j,j}<C_{\max}$. Under Assumptions 1 and 2 and the null hypothesis of homoskedasticity, the standard arguments of a normal distribution and chi-square distribution yield 
\begin{equation*}
P\left(\frac{1}{n}\|X^T\epsilon\|_\infty >\lambda_0\right)=O\left(\frac{1}{\sqrt{\log p}}\right).
\end{equation*}
Recall that under Assumption 2 and $s_0=o(\sqrt{n}/(\log p)^2)$, the compatibility condition holds with probability tending to one. Therefore, the standard arguments of the Lasso yield
\begin{equation*}
    \|\betahat-\beta_0\|_1=O_p\left(\sqrt{\frac{s_0^2\log p}{n}}\right),\ \ \frac{1}{n}\|X(\betahat-\beta_0)\|^2=O_p\left(\frac{s_0\log p}{n}\right),
\end{equation*}
under Assumptions 1 and 2 and the null hypothesis of homoskedasticity. See Lemmas 6.1 and 6.3 and Theorem 6.1 in \cite{bv2011}. By H\"{o}lder's inequality, we obtain 
\begin{align*}
    \left|\frac{2}{n}\sum_{i=1}^n\epsilon_iX_i^T(\betahat-\beta_0)\right|\leq \frac{2}{n}\|X^T\epsilon\|_{\infty}\|\betahat-\beta_0\|_1\leq \frac{Cs_0\log p}{n},
\end{align*}
with probability tending to one, so we obtain 
\begin{equation*}
   \frac{2}{n}\sum_{i=1}^n\epsilon_iX_i^T(\betahat-\beta_0)=O_p\left(\frac{s_0\log p}{n}\right). 
\end{equation*}
Note that 
\begin{equation*}
    \frac{1}{n}\sum_{i=1}^n\{X_i^T(\betahat-\beta_0)\}^2=\frac{1}{n}\|X(\betahat-\beta_0)\|^2=O_p\left(\frac{s_0 \log p}{n}\right).
\end{equation*}
Therefore, we obtain
\begin{equation}
    \frac{1}{n}\sum_{i=1}^n \hat{\epsilon}_i^2=\frac{1}{n}\sum_{i=1}^n\epsilon_i^2+O_p\left(\frac{s_0 \log p}{n}\right). \label{258}
\end{equation}
Because the first term on the right-hand side of \eqref{258} is $O_p(1)$ due to the law of large numbers, we obtain 
\begin{equation*}
    \left(\frac{1}{n}\sum_{i=1}^n \hat{\epsilon}_i^2\right)^2=\left(\frac{1}{n}\sum_{i=1}^n\epsilon_i^2\right)^2+O_p\left(\frac{s_0 \log p}{n}\right)+O_p\left(\frac{s_0^2 (\log p)^2}{n^2}\right).
\end{equation*}
Thus, we obtain 
\begin{equation}
    \left(\frac{1}{n}\sum_{i=1}^n \hat{\epsilon}_i^2\right)^2=\left(\frac{1}{n}\sum_{i=1}^n\epsilon_i^2\right)^2+o_p\left(\frac{1}{\sqrt{n}}\right) \label{259}.
\end{equation}
Next, we evaluate the sample mean of fourth-powered residuals \eqref{238}. By normality of errors and $p>\kappa n$ for some $0<\kappa<1$, we obtain  
\begin{align*}
    P\left(\max_{1\leq i\leq n}|\epsilon_i|>\sigma_0\sqrt{2 \log  p}\right)&\leq \sum_{i=1}^n P(|\epsilon_i|>\sigma_0\sqrt{2 \log p})\\
    &\leq \frac{1}{\sigma_0 \sqrt{\pi \log p}}\frac{n}{p}\\  
    & \leq \frac{1}{\sigma_0 \sqrt{\pi \log p}} \frac{1}{\kappa}\\
    &=O\left(\frac{1}{\sqrt{\log p}}\right).
\end{align*}
Therefore, the second term on the right-hand side of \eqref{238} is bounded by 
\begin{align*}
    \frac{6}{n}\sum_{i=1}^n\{\epsilon_iX_i^T(\betahat-\beta_0)\}^2&\leq  \frac{12\sigma_0^2\log p}{n}\sum_{i=1}^n\{X_i^T(\betahat-\beta_0)\}^2\leq \frac{Cs_0(\log p)^2}{n},
\end{align*}
with probability tending to one. Because $s_0=o(\sqrt{n}/(\log p)^2)$, we obtain 
\begin{equation}
\frac{6}{n}\sum_{i=1}^n\{\epsilon_iX_i^T(\betahat-\beta_0)\}^2=o_p\left(\frac{1}{\sqrt{n}}\right).    \label{267}
\end{equation}
Next, we evaluate the third term on the right-hand side of \eqref{238}. Because $X_{i,j}\sim N(0,\Sigma_{j,j})$ for $i=1,\ldots,n$, Gaussian tails yield 
\begin{align*}
P(\max_{i,j}|X_{i,j}|>2\sqrt{C_{\max}\log p})&\leq \sum_{i=1}^n\sum_{j=1}^pP(|X_{i,j}|>2\sqrt{C_{\max} \log p})\\
&\leq \sum_{j=1}^p \frac{n}{\sqrt{2 \pi \Sigma_{j,j}}}\frac{1}{\sqrt{C_{\max}\log p}}\exp\left(-\frac{2C_{\max}\log p}{\Sigma_{j,j}}\right)\\
&\leq \frac{1}{\sqrt{2 \pi C_{\min}}}\frac{1}{\sqrt{C_{\max}\log p}}\frac{n}{p}\\
&\leq \frac{1}{\sqrt{2 \pi C_{\min}}}\frac{1}{\sqrt{C_{\max}\log p}}\frac{1}{\kappa }= O\left(\frac{1}{\sqrt{\log p}}\right). 
\end{align*}
Thus, for $i=1,\ldots,n$, we obtain 
\begin{align*}
    |X_i^T(\betahat-\beta_0)|&\leq \sum_{j=1}^p|X_{i,j}||\betahat_j-\beta_{0j}|\\
    &\leq \max_{i,j}|X_{i,j}| \sum_{j=1}^p|\betahat_j-\beta_{0j}|\\
    &\leq 2\sqrt{C_{\max}\log p}\|\betahat-\beta_0\|_1,
\end{align*}
with probability tending to one. Therefore, the third term is bounded by
\begin{align*}
    \frac{1}{n}\sum_{i=1}^n\{X_i^T(\betahat-\beta_0)\}^4&\leq 4C_{\max}\log p\|\betahat-\beta_{0}\|_1^2\frac{1}{n}\sum_{i=1}^n\{X_i^T(\betahat-\beta_0)\}^2  \\
    &\leq \frac{Cs_0^3(\log p)^3}{n^2},
\end{align*}
with probability tending to one. Thus, we obtain 
\begin{align}
 \frac{1}{n}\sum_{i=1}^n\{X_i^T(\betahat-\beta_0)\}^4=o_p\left(\frac{1}{\sqrt{n}}\right). \label{282}
\end{align}
Next, we evaluate the fourth term on the right-hand side of \eqref{238}. Let $z=(\epsilon_1^3,\ldots, \epsilon_n^3)^T$. Then,
\begin{equation*}
    \frac{1}{n}\sum_{i=1}^n\epsilon_i^3X_i^T(\betahat-\beta_0)=\frac{1}{n}z^TX(\betahat-\beta_0),
\end{equation*}
and by H\"{o}lder's inequality 
\begin{equation*}
    \frac{1}{n}|z^TX(\betahat-\beta_0)|\leq \frac{1}{n}\|X^Tz\|_\infty \|\betahat-\beta_0\|_1.
\end{equation*}
Note that
\begin{align*}
    &P\left(\frac{1}{\sqrt{n}}\|X^Tz\|_\infty>\left\{\frac{\Sigma_{j,j}}{n}\sum_{i=1}^n\epsilon_i^6\right\}^{1/2} \sqrt{2\log p} \right)\\
    &=P\left(\frac{1}{\sqrt{n}}\max_{1\leq j \leq p}|X_{(j)}^Tz|>\left\{\frac{\Sigma_{j,j}}{n}\sum_{i=1}^n\epsilon_i^6\right\}^{1/2} \sqrt{2\log p}\right)  \\
    &\leq\sum_{j=1}^p P\left(\frac{1}{\sqrt{n}}|X_{(j)}^Tz|>\left\{\frac{\Sigma_{j,j}}{n}\sum_{i=1}^n\epsilon_i^6\right\}^{1/2} \sqrt{2\log p}\right)\\
    &= \sum_{j=1}^pP\left(\frac{1}{\sqrt{n}}\left|\sum_{i=1}^nX_{i,j}\epsilon_i^3\right|>\left\{\frac{\Sigma_{j,j}}{n}\sum_{i=1}^n\epsilon_i^6\right\}^{1/2} \sqrt{2\log p}\right) \\
    &=E\left[ \sum_{j=1}^pP\left(\frac{1}{\sqrt{n}}\left| \sum_{i=1}^nX_{i,j}\epsilon_i^3\right| >\left\{\frac{\Sigma_{j,j}}{n}\sum_{i=1}^n\epsilon_i^6\right\}^{1/2} \sqrt{2\log p}\Bigr| \epsilon \right)\right],
\end{align*}
where $X_{(j)}$ denotes the $j$th column of $X$. Because $X$ and $\epsilon$ are independent under the null hypothesis and  $X_{i,j} \sim N(0,\Sigma_{j,j})$ independently for $i=1,\ldots n$, we obtain
\begin{align*}
\frac{1}{\sqrt{n}}\sum_{i=1}^nX_{i,j}\epsilon_i^3|\epsilon \sim N\left(0,\Sigma_{j,j}\frac{1}{n}\sum_{i=1}^n\epsilon_i^6 \right),    
\end{align*}
 and 
 \begin{align*}
\left(\Sigma_{j,j} \frac{1}{n}\sum_{i=1}^n\epsilon_i^6\right)^{-1/2}\frac{1}{\sqrt{n}}\sum_{i=1}^nX_{i,j}\epsilon_i^3|\epsilon \sim N(0,1).     
 \end{align*}
Therefore, we obtain  
\begin{align*}
 &P\left(\frac{1}{\sqrt{n}}\left| \sum_{i=1}^nX_{i,j}\epsilon_i^3\right| >\left\{\frac{\Sigma_{j,j}}{n}\sum_{i=1}^n\epsilon_i^6\right\}^{1/2} \sqrt{2\log p}\Bigr| \epsilon \right)\\
    &\leq 2P\left(\frac{1}{\sqrt{n}} \sum_{i=1}^nX_{i,j}\epsilon_i^3 >\left\{\frac{\Sigma_{j,j}}{n}\sum_{i=1}^n\epsilon_i^6\right\}^{1/2}\sqrt{2\log p}\Bigr| \epsilon \right)\\
    &=2P\left(\left\{\Sigma_{j,j}\frac{1}{n}\sum_{i=1}^n\epsilon_i^6\right\}^{-1/2}\frac{1}{\sqrt{n}} \sum_{i=1}^nX_{i,j}\epsilon_i^3 > \sqrt{2\log p}\Bigr| \epsilon \right)\leq \frac{1}{\sqrt{\pi \log p}}\frac{1}{p},
\end{align*}
and 
\begin{equation*}
   E\left[ \sum_{j=1}^pP\left(\frac{1}{\sqrt{n}}\left| \sum_{i=1}^nX_{i,j}\epsilon_i^3\right| >\left\{\frac{\Sigma_{j,j}}{n}\sum_{i=1}^n\epsilon_i^6\right\}^{1/2} \sqrt{2\log p}\Bigr| \epsilon \right)\right]\leq \frac{1}{\sqrt{\pi \log p}}.
\end{equation*}
Now, we obtain 
\begin{equation*}
    P\left(\frac{1}{\sqrt{n}}\|X^Tz\|_\infty>\left\{\frac{\Sigma_{j,j}}{n}\sum_{i=1}^n\epsilon_i^6\right\}^{1/2} \sqrt{2\log p} \right)=O\left(\frac{1}{\sqrt{\log p}}\right).
\end{equation*}
Thus, we obtain
\begin{align*}
     \frac{1}{n}\|X^Tz\|_\infty \|\betahat-\beta_0\|_1\leq \left\{\frac{\Sigma_{j,j}}{n}\sum_{i=1}^n\epsilon_i^6\right\}^{1/2}\sqrt{\frac{2\log p}{n}}\|\betahat-\beta_0\|_1\leq \frac{Cs_0 \log p }{n},
\end{align*}
with probability tending to one. Therefore, the convergence rate of the fourth term is given by
\begin{equation}
    \frac{1}{n}\sum_{i=1}^n\epsilon_i^3X_i^T(\betahat-\beta_0)=o_p\left(\frac{1}{\sqrt{n}}\right). \label{318}
\end{equation}
Finally, we evaluate the fifth term on the right-hand side of \eqref{238}. The above arguments yield 
\begin{align*}
    4 \left|\frac{1}{n}\sum_{i=1}^n\epsilon_i\{X_i^T(\betahat-\beta_0)\}^3\right|&\leq 4 \frac{1}{n}\sum_{i=1}^n|\epsilon_i\{X_i^T(\betahat-\beta_0)\}^3| \nonumber \\
    &\leq 8\sqrt{2C_{\max}}\sigma_0 \log p\|\betahat-\beta_0\|_1 \frac{1}{n}\sum_{i=1}^n\{X_i^T(\betahat-\beta_0)\}^2 \nonumber \\
    &\leq \frac{Cs_0^2(\log p)^{5/2}}{n^{3/2}},
\end{align*}
with probability tending to one. Therefore, we obtain 
\begin{align}
\frac{4}{n}\sum_{i=1}^n\epsilon_i\{X_i^T(\betahat-\beta_0)\}^3=o_p\left(\frac{1}{\sqrt{n}}\right)\label{324}.
\end{align}
As a result, using \eqref{267}, \eqref{282}, \eqref{318}, and \eqref{324}, the equation \eqref{238} can be rewritten as  
\begin{equation*}
    \frac{1}{n}\sum_{i=1}^n\hat{\epsilon}_i^4=\frac{1}{n}\sum_{i=1}^n\epsilon_i^4+o_p\left(\frac{1}{\sqrt{n}}\right).
\end{equation*}
Thus, the test statistic $T$ can be expressed as 
\begin{align*}
T&=\frac{\frac{1}{n}\sum_{i=1}^n\hat{\epsilon}_i^4}{(\frac{1}{n}\sum_{i=1}^n\hat{\epsilon}_i^2)^2}-1\\
&=\frac{\frac{1}{n}\sum_{i=1}^n\epsilon_i^4+o_p\left(\frac{1}{\sqrt{n}}\right)}{(\frac{1}{n}\sum_{i=1}^n\epsilon_i^2)^2+o_p\left(\frac{1}{\sqrt{n}}\right)}-1\\
&=\frac{\frac{1}{n}\sum_{i=1}^n\epsilon_i^4}{(\frac{1}{n}\sum_{i=1}^n\epsilon_i^2)^2}-1+o_p\left(\frac{1}{\sqrt{n}}\right).
\end{align*}
Recalling \eqref{205}, we obtain 
\begin{equation*}
    \sqrt{n}(T-2)=\sqrt{n}(f(\hat{\theta})-f(\theta_0))+o_p(1).
\end{equation*}
\eqref{203} and \eqref{217} yield
\begin{equation*}
    \sqrt{n}(T-2)\xrightarrow{d}N(0,24).
\end{equation*}
This completes the proof.
\end{proof}

\bibliographystyle{apalike} 
\bibliography{bib}

\end{document}